\documentclass{amsart}

\usepackage{bbm}

\newcommand{\bR}{\mathbb{R}}

\usepackage{lmodern}
\usepackage[dvipsnames,svgnames,x11names,hyperref]{xcolor}
\usepackage{mathtools,url,graphicx,verbatim,amssymb,enumerate,stmaryrd,nicefrac}
\usepackage[pagebackref,colorlinks,citecolor=Mahogany,linkcolor=Mahogany,urlcolor=Mahogany,filecolor=Mahogany]{hyperref}
\usepackage{microtype}
\usepackage{tikz, tikz-cd}
\usetikzlibrary{matrix,calc,arrows,patterns,decorations.pathreplacing, decorations.markings}
\usepackage{caption}
\AtBeginDocument{%
   \def\MR#1{}}

\newtheorem{theorem}{Theorem}

\newtheorem*{corollary*}{Corollary}

\theoremstyle{definition}

\theoremstyle{remark}

\newtheorem{remark}[theorem]{Remark}
\newtheorem*{remark*}{Remark}

\newtheorem{question}[theorem]{Question}

\newcommand{\mr}[1]{{\rm #1}}

\tikzset{
  symbol/.style={
    draw=none,
    every to/.append style={
      edge node={node [auto=false]{$#1$}}}
  }
}

\title{There is no topological Fulton--MacPherson compactification}

\author{Alexander Kupers}
\email{a.kupers@utoronto.ca}
\address{Department of Computer and Mathematical Sciences \\ 
	University of Toronto Scarborough \\
	1265 Military Trail, Toronto, ON M1C 1A4 \\
	Canada}

\begin{document}

\begin{abstract}In this note, we prove that the Fulton--MacPherson compactification of configuration spaces of smooth manifolds can not be extended to topological manifolds in a natural manner, using recent work of Chen and Mann.\end{abstract}

\maketitle 

For any manifold $M$, the configuration space of $k$ ordered points in $M$ is given by
\[\mr{Conf}_k(M) \coloneqq \{(m_1,\ldots,m_k) \mid m_i \neq m_j \text{ if $i \neq j$}\} \subset M^k.\]

If $M$ is smooth, there exist \emph{Fulton--MacPherson compactifications} $\mr{FM}_k(M)$ for $k \geq 0$ \cite{Sinha} (first defined in \cite{AxelrodSinger}) with the following properties:
\begin{enumerate}[(i)]
	\item $\mr{FM}_k(M)$ is a smooth manifold with corners.
	\item The interior of $\mr{FM}_k(M)$ is $\mr{Conf}_k(M)$.
	\item The action of diffeomorphisms of $M$ on $\mr{Conf}_k(M)$ extends to $\mr{FM}_k(M)$.
	\item $\partial \mr{FM}_2(M)$ is diffeomorphic to the sphere bundle $STM$ of the tangent bundle, equivariantly for diffeomorphisms of $M$.
\end{enumerate}
These are but a few of its useful properties. For example, the collection of spaces $\{\mr{FM}_k(\bR^d)\}_{k \geq 0}$ can be assembled to an operad so that $\{\mr{FM}_k(M)\}_{k \geq 0}$ forms a right module over it. This is used in the proofs of various formality results, and constructions of configuration space integrals \cite{Kontsevich}. To understand the dependence of such results on a smooth structure, we are led to ask:

\begin{question}Does the Fulton--MacPherson compactification extend to topological manifolds?\end{question}

The following result tells us the answer is negative for any sufficiently natural extension.

\begin{theorem} When $M$ is a connected non-empty smooth manifold of dimension $d \geq 2$, the action of homeomorphisms of $M$ on $\mr{Conf}_2(M)$ does not extend to $\mr{FM}_2(M)$ compatibly with the action of diffeomorphisms.\end{theorem}

\begin{proof}If it did, we would get an action of $\mr{Homeo}(M)$ on $\partial \mr{FM}_2(M)$. By (ii), this is an extension of the transitive action of diffeomorphisms of $M$ on the $(2d-1)$-dimensional manifold $STM$. Hence there would be a transitive action of $\mr{Homeo}(M)$ on $STM$. By \cite[Theorem 1.1]{ChenMann}, the orbits of $\mr{Homeo}(M)$ acting on any finite-dimensional CW-complex must be covers of configuration spaces of $M$. These have dimension $rd$ for $r \geq 0$, and since $rd \neq 2d-1$ when $d \geq 2$ we get a contradiction.\end{proof}

\begin{remark}One could ask whether there exists an extension of the Fulton--MacPherson compactification to PL manifolds. We believe this is unlikely, but the above argument does not apply since \cite{ChenMann} does not apply to PL homeomorphisms.\end{remark}

\begin{remark}This argument also proves that in dimension $\geq 2$ it is impossible to give a construction of the spherical tangent bundle of a topological manifold which is natural in homeomorphisms.\end{remark}

\begin{remark}For $d=1$, such compactifications do exist. For example, $\mr{FM}_k(S^1) \cong S^1 \times W_k$ with $W_k$ an cyclohedron \cite[p.~5249]{BottTaubes} and the action on $\mr{FM}_k(S^1)$ by the group of diffeomorphism of the circle extends to its group of homeomorphisms $\mr{Homeo}(S^1) \cong \mr{O}(2) \ltimes \mr{Homeo}_\partial([0,1])$.\end{remark}

\bibliographystyle{amsalpha}
\bibliography{../refs2}

\providecommand{\bysame}{\leavevmode\hbox to3em{\hrulefill}\thinspace}
\providecommand{\MR}{\relax\ifhmode\unskip\space\fi MR }
\providecommand{\MRhref}[2]{%
  \href{http://www.ams.org/mathscinet-getitem?mr=#1}{#2}
}
\providecommand{\href}[2]{#2}
\begin{thebibliography}{Kon94}

\bibitem[AS94]{AxelrodSinger}
S.~Axelrod and I.~M. Singer, \emph{Chern-{S}imons perturbation theory. {II}},
  J. Differential Geom. \textbf{39} (1994), no.~1, 173--213. \MR{1258919}

\bibitem[BT94]{BottTaubes}
R.~Bott and C.~Taubes, \emph{On the self-linking of knots}, vol.~35, 1994,
  Topology and physics, pp.~5247--5287. \MR{1295465}

\bibitem[CM19]{ChenMann}
L.~Chen and K.~Mann, \emph{Structure theorems for actions of homeomorphism
  groups}, 2019, arXiv:1902.05117.

\bibitem[Kon94]{Kontsevich}
M.~Kontsevich, \emph{Feynman diagrams and low-dimensional topology}, First
  {E}uropean {C}ongress of {M}athematics, {V}ol. {II} ({P}aris, 1992), Progr.
  Math., vol. 120, Birkh\"{a}user, Basel, 1994, pp.~97--121. \MR{1341841}

\bibitem[Sin04]{Sinha}
D.~P. Sinha, \emph{Manifold-theoretic compactifications of configuration
  spaces}, Selecta Math. (N.S.) \textbf{10} (2004), no.~3, 391--428.
  \MR{2099074}

\end{thebibliography}

\vspace{+0.2cm}
\end{document}